\documentclass[a4paper,reqno,12pt]{amsart}
\usepackage{amsmath,amstext,amssymb,amsopn,amsthm,mathrsfs}
\usepackage{xcolor}
\usepackage[matrix,arrow,ps]{xy}
 
\numberwithin{equation}{section}
\allowdisplaybreaks

\newtheorem{theorem}{Theorem}[section]

\newtheorem{lemma}[theorem]{Lemma}

\newcommand{\N}{\mathbb{N}}

\newcommand{\R}{\mathbb{R}}

\newcommand{\Nn}{\mathbb{N}^n}

\newcommand{\Rn}{\mathbb{R}^n}

\newcommand*\calH{\mathcal{H}}

\newcommand*\calY{\mathcal{Y}}
\newcommand*\calZ{\mathcal{Z}}
\newcommand{\D}{{\rm Dom}\,}
\newcommand{\Cc}{C^\infty_c(\mathbb R^n)}

\def\a{\alpha}
\def\b{\beta}
\def\v{\varphi}

\def\ve{\varepsilon}

\begin{document}
\title[Self-adjoint realization of the harmonic oscillator in polar coordinates]	{Self-adjoint realization of the harmonic oscillator in polar coordinates and some consequences}

\author[K. Stempak]{Krzysztof Stempak}
\address{Krzysztof Stempak, 55-093 Kie\l{}cz\'ow, Poland}    
\email{krz.stempak@gmail.com}

\begin{abstract} 
We consider spectral decomposition of the harmonic oscillator in $\mathbb R^n$ in terms of two different orthonormal bases in $L^2(\mathbb R^n)$ consisting of its 
eigenfunctions. Then, using purely functional analysis tools we provide simple proofs of rotational symmetry of the Hermite projection operators studied by Kochneff, and  Thangavelu's Hecke-Bochner type identity.
\end{abstract}


\maketitle

\section{Introduction} \label{sec:1}
The harmonic oscillator 
\begin{equation*}
H =-\Delta + |x|^2 
\end{equation*}
is a model example of an unbounded operator on $L^2(\Rn)$ with discrete spectrum, whose spectral theory is completely understood.
Analysis of $H$, which is an important operator in mathematical physics (known as the quantum oscillator), was performed in numerous papers and monographs, 
see for instance, \cite{RS}, \cite{F2} or \cite{T2}. Initially considered with domain $\D H:=\Cc$, the operator $H$ is symmetric and nonnegative. Additionally, 
$H$ is essentially self-adjoint (this is a consequence of a general theorem, see \cite[Theorem X.28]{RS}), which means that its self-adjoint extension is unique.

The  \textit{multi-dimensional Hermite functions} $\{h_\a\colon \a=(\a_1,\ldots,a_n)\in\Nn\}$, 
\begin{equation*}
h_\a(x) =\prod_{i=1}^n h_{\a_i}(x_i), \qquad x = (x_1,\ldots, x_n)\in\Rn,
\end{equation*}
form an orthonormal basis in $L^2(\Rn)$. In addition, $h_\a$ are eigenfunctions of the differential operator $-\Delta + |x|^2$,
\begin{equation*}
\big(-\Delta+ |x|^2\big)h_\a = \lambda_\a h_\a, \qquad \lambda_\a = n+2|\a|,
\end{equation*}
where $|\a| = \a_1+\ldots+\a_n$ stands for the length of $\a\in\Nn$. Hence, a well known procedure (see \cite[Lemma 1.2.2]{Dav}) shows that 
$\calH$ defined by
\begin{align*}
\D\calH&=\{f\in L^2(\Rn)\colon \sum_{\a\in \Nn} |\lambda_\a\langle f,h_\a\rangle|^2 < \infty\},\\
\calH f&=  \sum_{\a\in \Nn}\lambda_\a \langle f, h_\a\rangle h_\a, \qquad f\in \D\calH,
\end{align*} 
is self-adjoint and its spectrum  is discrete and equals $\{n + 2k\colon k\in\mathbb N\}$. Since $\Cc\subset \D\calH$, it follows that $\calH$ is a self-adjoint 
extension of $H$.

Less known is a different realization of $\calH$ given in terms of another orthonormal basis of $L^2(\Rn)$ consisting of eigenfunctions of  $-\Delta + |x|^2$. 
See Sommen \cite{So} or Coulembier et al. \cite{CBS}. The basis in question is
\footnote{\textcolor[rgb]{1,0,0}{$\heartsuit$} If $n=1$, then $s\in\N$ is replaced by $s\in\{0,1\}$. This replacement is assumed throughout; see  Section~\ref{sec:4}, (D).}
$$
\phi_{k,s,j}(x):=\ell^{\frac n2-1+s}_k(|x|)Y_{s,j}(x),\qquad k\in\N,\quad s\in\N, \quad j=1,\ldots,d_s.
$$
Here $\ell^\beta_k$ stands for the $k$th Laguerre function (of convolution type) of order $\beta>-1$, $\{Y_{s,j}\}_{j=1}^{d_s}$ is a fixed orthonormal basis 
in the space of the solid harmonic polynomials homogeneous of degree $s$ in $\Rn$,  and $d_s$ is the dimension of this space. Additionally, one has
\begin{equation}\label{bb}
(-\Delta + |x|^2)\phi_{k,s,j}=\lambda_{k,s}\phi_{k,s,j}, \qquad \lambda_{k,s}=n+2(s+2k).
\end{equation}
Therefore, $\widetilde{\calH}$ defined by
\begin{align*}
\D\widetilde{\calH}&=\{f\in L^2(\Rn)\colon \sum_{k,s,j} |\lambda_{k,s}  \langle f,\phi_{k,s,j}  \rangle|^2 < \infty\},\\
\widetilde{\calH} f&=  \sum_{k,s,j}\lambda_{k,s}\langle f, \phi_{k,s,j}\rangle \phi_{k,s,j}, \qquad f\in \D\widetilde{\calH},
\end{align*} 
is a self-adjoint operator on $L^2(\Rn)$. It is easily seen that $\Cc\subset \D\,\widetilde{\calH}$ (see Section \ref{sec:4}, (A), for details) and hence 
$\widetilde{\calH}$ is an extension of $H$. Since $H$ is essentially self-adjoint it follows that $\widetilde{\calH}=\calH$. This important equality has 
some interesting consequences which we discuss below. 

We mention that the system $\{\phi_{k,s,j}\}$ was used by Ciaurri and Roncal \cite{CR} to define and investigate a Riesz transform for the harmonic 
oscillator in the setting of polar coordinates. 

To allow the reader concentrate on the main line of thoughts we decided to put some explanatory facts in the Appendix, Section \ref{sec:4}.

\textbf{Notation}. We shall write $\langle\cdot,\cdot\rangle_{L^2(X)}$ to denote the canonical inner product in $L^2(X)$, but for $X=\Rn$ 
with Lebesgue measure we shall skip the relevant subscript writing simply $\langle\cdot,\cdot\rangle$. The symbol $\sigma_{n-1}$ will stand 
for the surface measure on the unit sphere  $\Sigma_{n-1}=\{|x|=1\}$ in $\Rn$, so that Lebesgue measure in $\Rn$ is given in polar coordinates 
by $dx=r^{n-1}drd\sigma_{n-1}(x')$. Throughout, writing $x=rx'$ where $r=|x|$ and $x'=x/|x|$, will denote the representation of  $0\neq x\in\Rn$ 
in polar coordinates. Finally, $\N=\{0,1,\ldots\}$ and $\lfloor \cdot \rfloor$ will denote the floor function.

\section{Preliminaries} \label{sec:2}

In this section, to make this note  self-contained, we first collect necessary facts on the systems of Laguerre functions and spherical harmonics.
We refer to Section \ref{sec:4}, (D), where the case of spherical harmonics in dimension one is discussed separately.

Let $\b>-1$. The Laguerre functions
$$
\ell^\b_k(r)=\Big(\frac{2\Gamma(k+1)}{\Gamma(k+\b+1)}\Big)^{1/2}L^\b_k(r^2)e^{-r^2/2}, \qquad k\in\N,\quad r>0,
$$
where $L_k^\b$ stands for the Laguerre polynomial of order $\b>-1$ and degree $k$ (see \cite[p.76]{L}), form an orthonormal basis in 
$L^2\big(r^{2\b+1}\big):=L^2\big((0,\infty),r^{2\b+1}dr\big)$ and satisfy
\begin{equation}\label{Lag}
\Big(-\frac{d^2}{dr^2}-\frac{2\b+1}{r}\frac d{dr}+r^2\Big)\ell^\b_k=2(2k+\b+1)\ell^\b_k.
\end{equation}

A comprehensive presentation of the theory of spherical harmonics can be found in \cite[Chapter IV]{SW} or \cite[Chapter 2, H]{F}. Let $n\ge1$ be fixed. 
We apply the convention that if $Y$ is a solid harmonic in $\Rn$, then its restriction to $\Sigma_{n-1}$ will be denoted $\calY$ and called the 
spherical harmonic corresponding to $Y$. $H_{(s)}$ will stand for the space of solid harmonics homogeneous of degree $s\in\N$ in $\Rn$.  
Thus, if $Y\in H_{(s)}$, then $Y(rx')=r^s\calY(x')$. We shall write $\calH_{(s)}$ for the space of restrictions of $Y\in H_{(s)}$ to $\Sigma_{n-1}$; 
$\calH_{(s)}$ is a finite dimensional subspace of the Hilbert space $L^2(\Sigma_{n-1}):=L^2(\Sigma_{n-1},\sigma_{n-1})$. 

Let $d_{s}:={\rm dim}\,H_{(s)}={\rm dim}\,\calH_{(s)}$. Recall (see the proof of
\cite[(2.55) Corollary]{F} with slightly different notation) that $d_s={\rm dim}\,P_s-{\rm dim}\,P_{s-2}$, where $P_s$ stands for the space of homogeneous  polynomials of degree $s$ in $\Rn$. For any orthonormal basis $\calY_{s,1},\ldots,\calY_{s,d_s}$ of $\calH_{(s)}$ and $x',y'\in \Sigma_{n-1}$ it holds (see \cite[(2.57) Theorem, a) and b)]{F})
\begin{equation}\label{zon}
\sum_{j=1}^{d_s}\calY_{s,j}(x')\overline{\calY_{s,j}(y')}=\calZ^{x'}_s(y').
\end{equation}
Here $\calZ^{x'}_s\in \calH_{(s)}$ is the \textit{zonal harmonic of degree} $s$ \textit{with pole at} $x'$, which means that  for all $\calY\in \calH_{(s)}$
it holds $\calY(x')=\langle \calY,\calZ^{x'}_s\rangle_{L^2(\Sigma_{n-1})}$.

The following result is known and the proof is given only for completeness. 
\begin{lemma} \label{lem:ell}
The system $\{\phi_{k,s,j}\}$ is an orthonormal basis in $L^2(\Rn)$ and \eqref{bb} holds.
\end{lemma}
\begin{proof}
Using integration in polar coordinates, orthogonality and the normalization of the system  easily follows from the fact that for every 
$s\in \N$, $\{\ell^{\frac n2-1+s}_k\colon k\in\N\}$ and $\{Y_{s,j}\colon s\in\N, 1\le j\le d_s\}$ are orthonormal systems in $L^2(r^{2(\frac n2-1+s)+1})$ 
and $L^2(\Sigma_{n-1})$, respectively. We note that for $n=1$  integration in `polar coordinates' takes a special form, see Section~\ref{sec:4}, (D); this remark also applies in the remaining part of the proof. 

It remains to verify completeness. For this we check that Parseval's identity
\begin{equation}\label{aa}
\|f\|_{L^2(\Rn)}^2=\sum_{k,s,j}\big|\langle f,\phi_{k,s,j}\rangle \big|^2, \qquad f\in  L^2(\Rn),
\end{equation}
is satisfied; this will be a simple consequence of Parseval's identities for the orthonormal bases $\{\ell^{\b}_k\colon k\in\N\}$ and  
$\{\calY_{s,j}\colon s\in\N, 1\le j\le d_s\}$ in $L^2(r^{2\b+1})$ and  $L^2(\Sigma_{n-1})$, respectively. Indeed, denoting $f_r(x')=f(rx')$, $x'\in \Sigma_{n-1}$, we obtain
\begin{align*}
\int_{\Rn}|f|^2dx=\int_0^\infty\int_{\Sigma_{n-1}}|f(rx')|^2d\sigma(x')r^{n-1}dr&=
\int_0^\infty\sum_{s,j}|\langle f_r,\calY_{s,j}\rangle_{L^2(\Sigma_{n-1})}|^2r^{n-1}dr\\
&=\sum_{s,j}\int_0^\infty|\langle f_r,\calY_{s,j}\rangle_{L^2(\Sigma_{n-1})}|^2r^{n-1}dr.
\end{align*}
With $s$ and $j$ fixed, consider the function $F_{s,j}(r)=\langle f_r,\calY_{s,j}\rangle_{L^2(\Sigma_{n-1})}$ and write
$$
F_{s,j}(r)=r^sG_{s,j}(r) \quad {\rm with}\quad G_{s,j}(r)=r^{-s}F_{s,j}(r).
$$
Then 
\begin{align*}
\int_0^\infty|\langle f_r,\calY_{s,j}\rangle_{L^2(\Sigma_{n-1})}|^2r^{n-1}dr=\int_0^\infty|F_{s,j}(r)|^2r^{n-1}dr&=\int_0^\infty|G_{s,j}(r)|^2r^{2(\frac n2-1+s)+1}dr\\
&=\sum_k \big|\langle G_{s,j},\ell_k^{\frac n2-1+s}\rangle_{L^2(r^{n-1+2s})}\big|^2.
\end{align*}
But 
\begin{align*}
\langle G_{s,j},\ell_k^{\frac n2-1+s}\rangle_{L^2(r^{n-1+2s})}
&=\int_0^\infty r^{-s}\int_{\Sigma_{n-1}} f(rx')\overline{\calY_{s,j}(x')}d\sigma(x')\ell_k^{\frac n2-1+2s}(r)r^{n-1+2s}dr\\
&=\int_0^\infty\int_{\Sigma_{n-1}} f(rx')\ell_k^{\frac n2-1+2s}(r)r^s\overline{\calY_{s,j}(x')}d\sigma(x')r^{n-1}dr\\
&=\langle f,\phi_{k,s,j}\rangle .
\end{align*}
Combining the above finally gives \eqref{aa}. 

Proving \eqref{bb} we shall use the differential properties of the Laguerre functions $\ell^\b_k$, see \eqref{Lag}, and harmonicity and homogeneity of $Y_{s,j}$.
We first note that for $F(x)=f(|x|)$ with $f(r):=\ell^{\frac n2-1+s}_k(r)$ and $Y(x):=Y_{s,j}(x)$, since $Y$ is harmonic, we obtain
$$
-\Delta(FY)=(-\Delta F)Y-2\nabla F\cdot\nabla Y.
$$
But $F$ is radial and hence, with $r=|x|$, 
\begin{align*}
(-\Delta+|x|^2)F(x)&=\Big(-\frac{d^2}{dr^2}-\frac{n-1}{r}\frac d{dr}+r^2\Big)f(r)\\
&=\Big(-\frac{d^2}{dr^2}-\frac{2(\frac n2-1+s)+1}{r}\frac d{dr}+r^2\Big)f(r)+\frac{2s}r\frac d{dr}f(r)\\
&=(n+2(s+2k))f(r)+\frac{2s}r\frac d{dr}f(r),
\end{align*}
where in the last step \eqref{Lag} was used. 
To conclude verification of  \eqref{bb} we observe that
$$
\frac{s}rf'(r)Y(x)=\nabla F(x)\cdot\nabla Y(x).
$$
Indeed, $\nabla F(x)=\frac{f'(r)}r x$ and $x\cdot\nabla Y(x)=sY(x)$ since $Y$ is  homogeneous of degree $s$.
\end{proof}

\section{Main results} \label{sec:3}

Let $\Pi_{n+2m}$ stand for the orthogonal projection operator corresponding to the eigenvalue $n+2m$ and associated to $\calH$, 
$$
\Pi_{n+2m}f=\sum_{|\a|=m}\langle f,h_\a\rangle h_\a, \qquad f\in L^2(\Rn),
$$
so that the spectral decomposition of $\calH$ is
$$
\calH f= \sum_{m=0}^\infty (n+2m)\Pi_{n+2m}f, \qquad f\in \D \calH.
$$
Clearly, the integral kernel of $\Pi_{n+2m}$ is 
$$
\Phi_m(x,y)=\sum_{|\a|=m}h_\a(x)h_\a(y), \qquad x,y\in\Rn.
$$ 

Analogously, for the parallel realization of $\calH$, which we denoted  $\widetilde{\calH}$, 
\begin{equation}\label{eq2}
\widetilde{\Pi}_{n+2m}f=\sum_{0\le k\le \lfloor m/2\rfloor }\sum_{1\le j\le d_{m-2k}} \langle f,\phi_{k,m-2k,j}\rangle \phi_{k,m-2k,j}, \qquad f\in L^2(\Rn),
\end{equation}
so that 
$$
\widetilde{\calH} f= \sum_{m=0}^\infty (n+2m)\widetilde{\Pi}_{n+2m}f, \qquad f\in \D \widetilde{\calH},
$$
holds. The integral kernel of $\widetilde{\Pi}_{n+2m}$, expressed in polar coordinates,  is 
\begin{align*}
\widetilde{\Phi}_m(rx',uy')&=\sum_{0\le k\le \lfloor m/2\rfloor }\sum_{1\le j\le d_{m-2k}}\phi_{k,m-2k,j}(x)\overline{\phi_{k,m-2k,j}(y)}\\
&=\sum_{0\le k\le\lfloor m/2\rfloor}\ell^{\frac n2-1+m-2k}_k(r)\ell^{\frac n2-1+m-2k}_k(u)(ru)^{m-2k}\sum_{1\le j\le d_{m-2k}}\calY_{m-2k,j}(x')\overline{\calY_{m-2k,j}(y')}\\
&=\sum_{0\le k\le \lfloor m/2\rfloor } \ell^{\frac n2-1+m-2k}_k(r)\ell^{\frac n2-1+m-2k}_k(u)(ru)^{m-2k}\calZ^{x'}_{m-2k}(y'),
\end{align*}
where $\calZ^{x'}_{m-2k}$ is the zonal harmonic of degree $m-2k$ with pole at $x'$; we used \eqref{zon} in the last step.

For any self-adjoint operator on a Hilbert space its spectral decomposition is uniquely determined, see e.g. \cite[Theorem 5.7]{Sch}. Therefore, since 
$\calH=\widetilde{\calH}$, we have 
\begin{equation}\label{eq}
\Pi_{n+2m}=\widetilde{\Pi}_{n+2m}, \qquad m\in\N; 
\end{equation}
notably $\Phi_m=\widetilde{\Phi}_m$. An explanation of \eqref{eq} based on elementary means is contained in Section \ref{sec:4}, (B). 

\subsection{Rotational symmetry of projection operators} \label{ssec:4}
This property, to be precise for $g\in SO(n)$ only, was proved by Kochneff \cite{K} and required some effort in the proof.
\begin{theorem}\label{thm:zero} $($\cite[Theorem 3.4]{K}$)$
Let $T_g f(x)=f(gx)$ for $g\in O(n)$. Then we have 
\begin{equation}\label{com}
\Pi_{n+2m}\circ T_g=T_g\circ  \Pi_{n+2m}\qquad m\in\N,\quad g\in O(n). 
\end{equation} 
\end{theorem}
But \eqref{com}  is just a simple consequence of the spectral theorem. More precisely, if on a Hilbert space a bounded operator $B$ commutes with 
a self-adjoint operator $S$ (for an unbounded $S$ this means the inclusion $BS\subset SB$), then $B$ commutes with all spectral projections from 
the spectral decomposition of $S$. See e.g. \cite [Proposition 5.15]{Sch}, where the mentioned result is included in a more general setting. 

In our framework, since 
\begin{equation}\label{comm} 
\calH\circ T_g=T_g\circ\calH, \qquad g\in O(n),
\end{equation}
therefore \eqref{com} holds. The above commutation is naturally expected because $T_g$ commutes with  $H$ so it should commute with the self-adjoint 
extension of $H$. However, it requires a formal proof which is outlined in Section~\ref{sec:4}, (C).

\subsection{Hecke-Bochner type identity for the Hermite projections} 
This identity was  proved by Thangavelu \cite{T1} (an earlier paper \cite{T0} contains its proof only for $n$ even). A shorter proof was provided by Kochneff 
\cite{K}. In both cases the proofs relied on appropriate evaluations of integrals with Hermite and Laguerre functions involved. 

After adjusting the present notation with that in \cite{T1}, the result is as follows.
\begin{theorem}\label{thm:one} $($\cite[Theorem 2.1]{T1}$)$
Let $f\in L^2(\Rn)$ be of the form $f(x)=f_0(|x|)Y(x)$, where $Y$ is a solid harmonic of homogeneity $M$. Then, for $K\in\N,$, 
\begin{equation}\label{Th}
\Pi_{n+2(M+2K)}f(x)=
\langle f_0,\ell^{\frac n2-1+M}_{K}\rangle_{L^2(r^{\frac n2-1+M})}\ell^{\frac n2-1+M}_{K}(|x|)Y(x),
\end{equation}
and $\Pi_{n+2m}f=0$ when $m$ is not of the form $m=M+2K$ for some $K\in\N$. 
\end{theorem}
Again, we refer to Section~\ref{sec:4} (D), where the case $n=1$ is commented separately.

To check \eqref{Th} by elementary means we use \eqref{eq} and note that for $g\in L^2(\Rn)$ of the form $g(x)=g_0(|x|)\hat Y(x)$, where $\hat Y$ is a solid harmonic of homogeneity $\hat M$, integrating in polar coordinates gives
\footnote{$\diamondsuit$ We slightly abuse the notation writing the inner product of $f_0$ and $g_0$ in $L^2\big(r^{2(\frac n2-1+\frac{M+\hat M}2)+1}\big)$ in place of $\int_{\Rn}f_0(r)\overline{g_0(r)} r^{n-1+M+\hat{M}}dr$. Notice also that the assumptions imposed on $f_0$ and $g_0$ imply that $r^{(n-1+M)/2}f_0\in L^2((0,\infty),dr)$ and $r^{(n-1+\hat M)/2}g_0\in L^2((0,\infty),dr)$, so $f_0\overline{g_0}\in  L^1(r^{n-1+M+\hat{M}})$.}
\begin{equation}\label{ble}
\langle f,g\rangle=\langle f_0,g_0\rangle_{L^2\big(r^{2(\frac n2-1+\frac{M+\hat M}2)+1}\big)}\langle \calY,\hat{\calY}\rangle_{L^2(\Sigma_{n-1})}
\end{equation}
and thus $\langle f,g\rangle=0$ when $M\neq\hat M$. Therefore, looking at \eqref{eq2} it is clear that given $m$, all $\langle f,\phi_{k,m-2k,j}\rangle$ vanish unless $m-2k=M$ for some $0\le k\le\lfloor m/2\rfloor $, and hence $\Pi_{n+2m}f=0$ when $m$ is not of the form $m=M+2K$ for some $K\in\N$. Now, let $m=M+2K$, 
$K\in\N$. Then, by \eqref{eq2}, \eqref{ble}, and using the fact that $\{\calY_{M,j}\colon 1\le j\le d_M\}$ is an orthonormal basis in $\calH_{(M)}$, we obtain
\begin{align*}
\widetilde{\Pi}_{n+2(M+2K)}f(x)&=\sum_{0\le k\le \lfloor M/2\rfloor +K}\sum_{1\le j\le d_{M+2(K-k)}} \langle f,\phi_{k,M+2(K-k),j}\rangle \phi_{k,M+2(K-k),j}(x)\\
&= \sum_{1\le j\le d_{M}} \langle f,\phi_{K,M,j}\rangle \phi_{K,M,j}(x)\\
&=\langle f_0,\ell^{\frac n2-1+M}_K\rangle_{L^2(r^{\frac n2-1+M})}\ell^{\frac n2-1+M}_K(|x|)\, |x|^M\sum_{1\le j\le d_{M}} \langle \calY,\calY_{M,j}\rangle_{L^2(\Sigma_{n-1})}\calY_{M,j}(x')\\
&=\langle f_0,\ell^{\frac n2-1+M}_K\rangle_{L^2(r^{\frac n2-1+M})}\ell^{\frac n2-1+M}_K(|x|)Y(x).
\end{align*}

\section{Appendix} \label{sec:4}
\noindent (A) To prove  $\Cc\subset \D\,\widetilde{\calH}$ it suffices to check that for any $\v\in \Cc$ and every $N\in\N$ it holds
\begin{equation}\label{est}
|\langle \v,\phi_{k,s,j}\rangle|\le C_{\v,N}\lambda_{k,s}^{-N}, 
\end{equation}
uniformly in $k,s\in \N$ and $j\in\{1,\ldots,d_s\}$. This reduces to checking that 
$$
\langle \v,(-\Delta+|x|^2)\phi_{k,s,j}\rangle=\langle (-\Delta+|x|^2)\v,\phi_{k,s,j}\rangle.
$$
Indeed, since $(-\Delta+|x|^2)\v\in \Cc$, we can repeat this argument to obtain 
$$
\lambda_{k,s}^N\langle \v,\phi_{k,s,j}\rangle=\langle (-\Delta+|x|^2)^N\v,\phi_{k,s,j}\rangle
$$ 
and then \eqref{est} follows. 

Let ${\rm supp}\,\v\subset \{|x|\le R-1\}$ for some $R>1$. Obviously, it suffices to verify that
\begin{equation*}
\langle \v,(-\Delta)\phi_{k,s,j}\rangle=\langle (-\Delta)\v,\phi_{k,s,j}\rangle.
\end{equation*}
But this follows from Green's formula, see e.g. \cite[Theorem D.9, (D.6), p.\,408]{Sch} because $\v$ and  
the directional outward normal derivative $\frac{\partial\v}{\partial\nu}$, vanish on the boundary of $\{|x|\le R\}$.

\vskip0.3cm
\noindent (B) It is certainly pedagogical to deliver a proof of \eqref{eq} by elementary means. For this, given $n=1,2,\ldots$, it suffices to check that 
for any $m\in\N$ the projection spaces
$$
{\rm lin}\,\{h_\a\colon |\a|=m\}\quad {\rm and} \quad {\rm lin}\,\{\phi_{k,m-2k,j}\colon 0\le k\le \lfloor m/2\rfloor , 1\le j\le d_{m-2k} \}
$$
coincide. This is equivalent to checking that the corresponding spaces of polynomials, 
$$
V_m={\rm lin}\,\{H_\a\colon |\a|=m\},
$$ 
and 
$$
\widetilde{V}_m= \quad {\rm lin}\,\{L_{k}^{\frac n2-1+m-2k}(|x|^2)Y_{m-2k,j}(x)\colon 0\le k\le \lfloor m/2\rfloor , 1\le j\le d_{m-2k}\},
$$
coincide. Since $V_m$ and $\widetilde{V}_m$ are finite dimensional in the linear space of all polynomials in $n$ variables, to reach the goal it suffices 
to verify that ${\rm dim}\,V_m={\rm dim}\,\widetilde{V}_m$ and to check $\widetilde{V}_m\subset V_m$, say. The dimension of $V_m$ is the dimension of $P_m$ 
(and equals $\frac{(n-1+m)!}{m!(n-1)!}$); cf. \cite[(2.54) Proposition]{F}. On the other hand, ${\rm dim}\,\widetilde{V}_m=d_m+d_{m-2}+d_{m-4}+\ldots$ 
with the last summand equal to $d_{1}$ or $d_{0}$, depending on the parity of $m$. But, see Section \ref{sec:2}, $d_i={\rm dim}\,P_i-{\rm dim}\,P_{i-2}$ 
(with the convention that ${\rm dim}\,P_{-2}={\rm dim}\,P_{-1}=0$), so that the relevant dimensions indeed coincide. 
For the inclusion choose $L_{k}^{\frac n2-1+m-2k}(|x|^2)Y_{m-2k,j}(x)$ with $0\le k\le \lfloor m/2\rfloor , 1\le j\le d_{m-2k}$. The degree of this polynomial 
is $2k+(m-2k)=m$. Since $H_\a$ with  $|\a|=m$ form an algebraic basis in $P_m$, the chosen function is a linear combination of $H_\a$ with $|\a|=m$. 

\vskip0.3cm
\noindent (C) The proof of \eqref{comm} requires yet another realization of $\calH$, see \cite{S}, where $\D\calH$ is realized as a Sobolev-type space 
and $\calH$ is defined in terms of a sesquilinear form. Then the proof  goes, \textit{mutatis mutandis}, as the proof of \cite[Proposition D.3]{S}, 
where $g$ was restricted to a finite reflection group, a subgroup of $O(n)$. This proof uses \cite[Lemmas D.2 and D.1]{S} and again, their proofs 
are easily adapted to the broader context of $g\in O(n)$.

\vskip0.3cm 
\noindent (D) The results discussed in this note include the case $n=1$ but this requires some comments. 
Considering $n=1$ let us begin with spherical harmonics for $\R$, cf. \cite[p.\,100]{F}. Then $\Sigma_0=\{-1,1\}$, $\sigma_0$ is the counting measure 
on $\{-1,1\}$, and the integration in `polar coordinates' then is $\int_\R f(x)\,dx=\int_0^\infty \sum_{\ve=\pm1}f(\ve r)\,dr$   . The space of solid harmonics is two-dimensional and spanned by 1 and $x$; more precisely, $H_{(0)}={\rm lin}\,\{1\}$, $H_{(1)}={\rm lin}\,\{x\}$, 
and $H_{(s)}=\{0\}$ for $s\ge2$. Moreover, $\phi_{k,0,1}=2^{-1/2}\ell^{-1/2}_k(|x|)$ and $\phi_{k,1,1}=2^{-1/2}\ell^{1/2}_k(|x|)x$, $k\in\N$, form an orthonormal basis in $L^2(\R)$ (here the subscript $s$ is limited to $s=0$ and $s=1$ only). They are eigenfunctions of $-\frac{d^2}{dx^2}+x^2$ with eigenvalues $1+4k$ 
and $3+4k$, respectively; this is easily checked by means of \eqref{Lag}. Thus the projection spaces corresponding to $\Pi_{1+2m}$, $m\in\N$, are one-dimensional and equal $V_{1+2\cdot 2k}={\rm lin}\,\{\phi_{k,0,1}\}$ for $m=2k$, and $V_{1+2(2k+1)}={\rm lin}\,\{\phi_{k,1,1}\}$ for $m=2k+1$, $k\in\N$. Consequently, 
the projection operators have very simple form 
\begin{equation}\label{oone}
\Pi_{1+2m}f=
\begin{cases}
\langle f,\phi_{k,0,1}\rangle \phi_{k,0,1},& m=2k,\\
\langle f,\phi_{k,1,1}\rangle \phi_{k,1,1},& m=2k+1.
\end{cases}
\end{equation}
Since for $n=1$ the admissible $M$ in Theorem \ref{thm:one} is limited  to $M=0$ or $M=1$ with (up to a multiplicative constant) $Y=1$ for $M=0$ or $Y=x$ 
for $M=1$, and $f(x)=f_0(|x|)Y(x)$ means that $f$ is even/odd on $\R$ for $M=0$ or $M=1$, one easily recovers in \eqref{oone} the equality contained in \eqref{Th}.


\end{document}